\documentclass[11pt,oneside,english,jou]{amsart}
\usepackage[T1]{fontenc}
\usepackage[latin9]{inputenc}
\usepackage{verbatim}
\usepackage{amsthm}
\usepackage{amssymb}

\makeatletter
\numberwithin{equation}{section}
\numberwithin{figure}{section}
\theoremstyle{plain}
\newtheorem{thm}{Theorem}[section]
  \theoremstyle{remark}
  \newtheorem{rem}[thm]{Remark}
  \theoremstyle{definition}
  \newtheorem{defn}[thm]{Definition}
  \theoremstyle{definition}
  \newtheorem{example}[thm]{Example}

  \theoremstyle{remark}
  
  \theoremstyle{definition}

  \theoremstyle{plain}
  \newtheorem{prop}[thm]{Proposition}
  \theoremstyle{definition}
  \newtheorem{problem}[thm]{Problem}
  \theoremstyle{plain}
  \newtheorem{lem}[thm]{Lemma}
  \theoremstyle{plain}
  \newtheorem{cor}[thm]{Corollary}

\newcommand{\eps}{\epsilon}

\newcommand{\N}{\mathbb{N}}

\makeatother

\usepackage{babel}

\begin{document}

\title[Dynamical Embedding in Cubical Shifts, TRP \& SBP]{Dynamical Embedding in Cubical Shifts \& the Topological Rokhlin and Small Boundary Properties}

\author{Yonatan Gutman}
\begin{abstract}
According to a conjecture of Lindenstrauss and Tsukamoto, a topological
dynamical system $(X,T)$ is embeddable in the $d$-cubical
shift $(([0,1]^{d})^{\mathbb{Z}},\ shift)$ if both its \textit{mean
dimension} and \textit{periodic dimension} are strictly bounded by
$\frac{d}{2}$. We verify the conjecture for the class of systems
admitting finite dimensional non-wandering sets (under the additional
assumption of closed periodic points set). The main tool in the proof
is the new concept of \textit{local markers}. Continuing the investigation
of (global) markers initiated in pervious work it is shown that the
\textit{marker property} is equivalent to a topological version of
the \textit{Rokhlin Lemma.} Moreover new classes of systems are found
to have the marker property, in particular, extensions of aperiodic
systems with a countable number of minimal subsystems. Extending work
of Lindenstrauss we show that for systems with the marker property
vanishing mean dimension is equivalent to the small boundary property.
Finally we answer a question by Downarowicz in the affirmative: the
small boundary property is equivalent to admitting a zero-dimensional
isomorphic extension.
\end{abstract}

\date{\today}
\thanks{The author was supported by the ERC Grant "Approximate Algebraic Structures and Applications".}
\maketitle

\section{Introduction}

The question under which conditions a topological dynamical system
$(X,T)$ is embeddable in the \textit{$d$-cubical shift} $(([0,1]^{d})^{\mathbb{Z}},\ shift)$
stems from Auslander's 1988 influential book. According to Jaworski's
Theorem (1974), for $X$ finite-dimensional and $T$ aperiodic, embedding
is possible with $d=1$. Auslander posed the question if for $d=1$, it is
sufficient that $X$ is minimal. The question was solved in the negative
by Lindenstrauss and Weiss (2000), adroitly using the invariant of
\textit{mean dimension} introduced by Gromov (1999). Around the same
time Lindenstruass (1999) showed that if $X$ is an extension of a
minimal system and $mdim(X,T)<\frac{d}{36}$, then $(X,T)$ is embeddable
in $(([0,1]^{d})^{\mathbb{Z}},\ shift)$. Recently Lindenstrauss and
Tsukamoto (2012) have introduced a unifying conjecture and several
cases of this conjecture have been verified. According to this conjecture
the only obstructions for embeddability are given by the invariants
of \textit{mean dimension} and \textit{periodic dimension,} the later
quantifying the natural obstruction due to the set of periodic points.
A precise statement of the conjecture is that $mdim(X,T)<\frac{d}{2}$
and $perdim(X,T)<\frac{d}{2}$ imply $(X,T)$ is embeddable in $(([0,1]^{d})^{\mathbb{Z}},\ shift)$.
In Gutman and Tsukamoto (2012) the conjecture was verified for extensions
of aperiodic subshifs and in Gutman (2012) the conjecture was verified
for finite-dimensional systems. In the same article it was shown that
for extensions of aperiodic finite-dimensional systems $mdim(X,T)<\frac{d}{16}$
implies $(X,T)$ is embeddable in $(([0,1]^{d+1})^{\mathbb{Z}},\ shift)$.

A keen observer will notice that all embedding results mentioned above,
involving infinite-dimensional systems, require the assumption of
aperiodicity. This is due to a common device used in the proofs: existence
of \textit{markers} (of all orders). Markers can be thought of as a
suitable generalization of the familiar markers of symbolic dynamics,
introduced by Krieger (1982), to the setting of arbitrary dynamical
systems. As a necessary condition for the existence of markers of
all orders is the aperiodicity of the system, one is confined to the
category of aperiodic systems.

In this article we resolve this difficulty by introducing the concept
of \textit{local markers.} It has the desired consequence of allowing
us to treat some infinite-dimensional systems admitting periodic points.
In particular we verify the Lindenstrauss-Tsukamoto Conjecture for
the class of systems whose non-wondering set is finite dimensional
(under the additional assumption of closed periodic points set) and discuss some examples. The result
is a consequence of a more general embedding theorem stating that
systems with the local marker property verifying $mdim(X,T)<\frac{d}{36}$
and $perdim(X,T)<\frac{d}{2}$ are embeddable in $(([0,1]^{d})^{\mathbb{Z}},\ shift)$.

Recognizing the importance of markers, both local and global, we
continue the investigation of markers as carried out in Gutman (2012)
which itself was a generalization of previous work by Bonatti and Crovisier (2004), and prove in particular that an aperiodic system with a countable number of minimal subsystems admits the marker property.

In Gutman (2011) the notion of the topological Rokhlin property was
introduced. This is a dynamical topological analogue of the Rokhlin
Lemma of measured dynamics. Here we show that the marker property
is equivalent to a strong version of the topological Rokhlin property.
Following closely Lindenstrauss (1999) this characterization results
with several fruitful applications: Systems with the marker property
admit a compatible metric with respect to which the \textit{metric
mean dimension} equals the (topological) mean dimension. Moreover
for such systems, vanishing mean dimension is equivalent to the small
boundary property and to being an inverse limit of finite
entropy systems. Using a different method we show that the small boundary
property is equivalent to admitting a zero-dimensional isomorphic
extension. This answers a question by Downarowicz in the affirmative.

\subsection*{Acknowledgements:}

I would like to thank Jérôme Buzzi, Sylvain Crovisier, Tomasz Downarowicz,
Misha Gromov, Mariusz Lema\'{n}czyk, Elon Lindenstrauss, Julien Melleray,
Micha\l{} Rams and Benjamin Weiss for helpful discussions. Special
thanks to Tomasz Downarowicz and the Wroc\l{}aw University of Technology
for hosting me for numerous times during the time I was working on
this article.

\section{Preliminaries}

The following article is closely related to the article \cite{Gut12a}
and we recommend the reader to familiarize herself or himself with
the Introduction and Preliminaries sections of that article.

\subsection{Conventions}

Throughout the article with the exception of Section \ref{sec:A-characterization-of SBP},
a topological dynamical system (t.d.s) $(X,T)$ consists of a \textit{metric}
compact space $(X,d)$ and a \textit{homeomorphism} $T:X\rightarrow X$.
$P=P(X,T)$ denotes the set of periodic points and $P_{n}$ denotes
the set of periodic points of period $\leq n$. In addition we use
the notation $H_{n}=P_{n}\setminus P_{n-1}$. $\triangle={\{(x,x)|\, x\in X\}}$
denotes the \emph{diagonal} of $X\times X$. If $x\in X$ and $\epsilon>0$,
let $B_{\epsilon}(x)=\{y\in X|\, d(y,x)<\epsilon\}$ denote the open
ball around $x$. We denote $\overline{B}_{\epsilon}(x)=\overline{B_{\epsilon}(x)}$.
Note $\overline{B}_{\epsilon}(x)\subseteq\{y\in X|\, d(y,x)\leq\epsilon\}$
but equality does not necessary hold. For $f,g\in (C(X,[0,1]^{d})$, we define $||f-g||_{\infty}\triangleq sup_{x\in X}||f(x)-g(x)||_{\infty}$.

\subsection{The Non-Wandering Set }

Let $(X,T)$ be a t.d.s. A point $x\in X$ is said to be \textbf{non-wandering
}if for every open set $x\in U$, there is $k\in\mathbb{Z}$ so that
$U\cap T^{k}U\neq\emptyset$. The \textbf{non-wandering set }$\Omega(X)$
is the collection of all non-wandering points. Note $\Omega(X)$ is a non-empty,
closed and $T$-invariant set.

The following three subsections follow closely the corresponding subsections
in \cite{Gut12a}:

\subsection{Dimension}

Let $\mathcal{C}$ denote the collection of open (finite) covers of
$X$. For $\alpha\in\mathcal{C}$ define its \textit{order} by $ord(\alpha)=\max_{x\in X}\sum_{U\in\alpha}1_{U}(x)-1$.
Let $D(\alpha)=\min_{\beta\succ\alpha}ord(\beta)$ (where $\beta$
\emph{refines} $\alpha$, $\beta\succ\alpha$, if for every $V\in\beta$,
there is $U\in\alpha$ so that $V\subset U$). The Lebesgue covering
dimension is defined by $dim(X)=\sup_{\alpha\in\mathcal{C}}D(\alpha)$.

\subsection{Periodic Dimension\label{subsec:perdim}}

Let $P_{m}$ denote the set of points of period $\leq m$. Introduce
the infinite vector $\overrightarrow{perdim}(X,T)=\big(\frac{dim(P_{m})}{m}\big)_{m\in\mathbb{N}}$.
This vector is clearly a topological dynamical invariant. Let $d>0$.
We write $perdim(X,T)<d$, if for every $m\in\mathbb{N}$, $\overrightarrow{perdim}(X,T)|_{m}<d$.

\subsection{Mean Dimension}

Define:

\[
mdim(X,T)=\sup_{\alpha\in\mathcal{C}}\lim_{n\rightarrow\infty}\frac{D(\alpha^{n})}{n}
\]
where $\alpha^{n}=\bigvee_{i=0}^{n-1}T^{-i}\alpha.$ Mean dimension
was introduced by Gromov \cite{G} and systematically investigated
by Lindenstrauss and Weiss in \cite{LW}.

\vskip 0.7 cm

The following two Subsections follow closely the corresponding Subsections
in \cite{G11}:
\subsection{The Topological Rokhlin Property\label{sub:Topological-Rokhlin-Property}}

\label{subsec:Rokhlin's Lemma} The classical Rokhlin lemma states
that given an aperiodic invertible measure-preserving system $(X,T,\mu)$
and given $\eps>0$ and $n\in\mathbb{N}$, one can find $A\subset X$
so that $A,TA,\ldots,T^{n-1}A$ are pairwise disjoint and $\mu\big(\bigcup_{k=0}^{n-1}T^{k}A\big)>1-\eps$.
It easily follows that given an aperiodic invertible measure-preserving
system $(X,T,\mu)$ and given $\eps>0$, one can find a measurable
function $f:X\rightarrow\{0,1,\ldots,n-1\}$ so that if we define
the \textit{exceptional set} $E_{f}=\{x\in X\,|\, f(Tx)\neq f(x)+1\}$,
then $\mu(E_f)<\eps$. The new formulation allows us to generalize to
the topological category. Indeed following \cite{SW}, given a t.d.s
$(X,T)$ and a set $E\subset X$, we define the orbit-capacity of
a set $E$ in the following manner (the limit exists):

\[
ocap(E)=\lim_{n\rightarrow\infty}\frac{1}{n}\sup_{x\in X}\sum_{k=0}^{n-1}1_{E}(T^{k}x)
\]
$(X,T)$ is said to have the \textbf{topological Rokhlin property
(TRP)} if and only if for every $\epsilon>0$ there exists a continuous
function $f:X\rightarrow\mathbb{R}$ so that for the \textit{exceptional
set} $E_{f}=\{x\in X\,|\, f(Tx)\neq f(x)+1\}$,
one has $ocap(E_{f})<\epsilon$.

\subsection{The Small Boundary Property\label{sub:The-Small-Boundary}}

\label{SBP section} Following \cite{SW} we call $E\subset X$ \textbf{small}
if $ocap(E)=0$. For closed sets this has a simple interpretation.
Indeed a closed set $A\subset X$ is small if and only if for any
$T$-invariant measure $\mu$ of $X$, one has $\mu(A)=0$. When $X$
has a basis of open sets with small boundaries, $(X,T)$ is said to
have the \textbf{small boundary property (SBP)}. In \cite{LW} it
was shown that SBP implies mean dimension zero. In \cite{G11} it
was shown that if $(X,T)$ is an extension of an aperiodic space with
SBP then it has TRP.

\subsection{The Metric Mean Dimension\label{sub:The-Metric-Mean}}

A set $S\subset X$ is called $(n,\epsilon,d)$-spanning if for every
$x\in X$ there is a $y\in S$ so that for all $0\leq k<n$, $d(T^{k}x,T^{k}y)<\epsilon$.
Define $A(n,\epsilon,d)$ to be the cardinality of a minimal $(n,\epsilon,d)$-
spanning set. Define:
\[
s(\epsilon,d)=\limsup_{n\rightarrow\infty}\frac{\log(A(n,\epsilon,d))}{n}
\]

\[
mdim_{d}(X,T)=\liminf_{\epsilon\rightarrow0}\frac{s(\epsilon,d)}{|\log(\epsilon)|}
\]

\noindent In \cite{LW} it was shown that $mdim_{d}(X,T)\leq mdim(X,T)$.
By a classical theorem of Bowen and Dinaburg, the topological entropy
is given by $h_{top}(X,T)~ =~\lim_{\epsilon\rightarrow0}s(\epsilon,d)$.
Thus it was concluded in \cite{LW} that finite topological entropy
implies mean dimension zero.

\subsection{Overview of the Article}

In Section \ref{sec:The-Marker-Property} the definition of the marker
property is recalled and new classes of system admitting the marker
property are exhibited, in particular, extensions of aperiodic systems
with a countable number of minimal subsystems. Additionally some simple examples are discussed. In Section \ref{sec:The-Local-Marker}
the local marker property is defined and verified for systems of finite dimensional systems
with closed sets of periodic points. In Section \ref{sec:The-Strong-Topological}
the local and global strong topological Rokhlin properties are introduced
and investigated. In particular it is shown that the marker property
is equivalent to the (global) strong topological Rokhlin property.
In Section \ref{sec:An-Embedding-Theorem} the following embedding
theorem is proven: If $(X,T)$ has the local marker property, $mdim(X,T)<\frac{d}{36}$
and $perdim(X,T)<\frac{d}{2}$, then $(X,T)$ is generically embeddable
in $(([0,1]^{d})^{\mathbb{Z}},\ shift)$. In Section \ref{sec:Applications}
various applications of the embedding theorem are given, in particular,
the verification of the Lindenstrauss-Tsukamoto Conjecture for the
class of systems admitting finite dimensional non-wandering sets (under
the additional assumption of closed periodic points set). Additionally some examples are constructed. In Section
\ref{sec:Equivalence-SBP-Mdim=00003D0} it is shown that systems
with the marker property admit a compatible metric with respect to
which the \textit{metric mean dimension} equals the (topological)
mean dimension. Moreover for systems with the marker property, vanishing
mean dimension is equivalent to the having the small boundary property
and to being an inverse limit of finite entropy systems. In Section
\ref{sec:A-characterization-of SBP} is is shown that the small boundary
property is equivalent to admitting a zero-dimensional isomorphic
extension. The Appendix contains auxiliary lemmas.

\section{The Marker Property\label{sec:The-Marker-Property} }
\begin{defn}
A subset $F$ of a t.d.s $(X,T)$ is called an \textbf{$n$-marker}
($n\in\mathbb{N}$) if:
\begin{enumerate}
\item $F\cap T^{i}(F)=\emptyset$ for $i=1,2,\ldots,n-1$.
\item The sets $\{T^{i}(F)\}_{i=1}^{m}$ cover $X$ for some $m\in\mathbb{N}$.
\end{enumerate}
The system $(X,T)$ is said to have the \textbf{marker property} if
there exist \textit{open} $n$-markers for all $n\in\mathbb{N}$.
\end{defn}

\begin{rem} \label{Rem: marker property stable under extension}Clearly
the marker property is stable under extension, i.e. if $(X,T)$ has
the marker property and $(Y,S)\rightarrow(X,T)$ is an extension,
then $(Y,S)$ has the marker property. \end{rem}

\begin{rem} By Lemma A.1 of \cite{Gut12a} $(X,T)$ has a closed
$n$-marker iff $(X,T)$ has an open $n$-marker. \end{rem}

The marker property was first defined in \cite{Dow06} (Definition
2), where one requires the $n$-markers to be clopen. In the same
article it was proven that an extension of an aperiodic zero-dimensional
(non necessarily invertible) t.d.s has the marker property. This was
essentially based on the \textquotedbl{}Krieger Marker Lemma\textquotedbl{}
(Lemma 2 of \cite{K82}). In \cite{Gut12a} Theorem 6.1 it was proven
that aperiodic finite dimensional t.d.s have the marker property.
From \cite[Lemma 3.3]{L99} it follows that an extension of an aperiodic
minimal system has the marker property. Given these results it is natural to ask the following
question:

\begin{problem} \label{Ques: Aperiodic-->Marker?}Does any aperiodic
system have the marker property? \end{problem}

We do not know the answer of the previous problem. However we are able to prove two theorems establishing the existence of the marker property under natural assumptions. We also discuss examples.
\begin{thm}
\label{thm:countable number of minimal systems-->Marker Property}(Downarowicz
\& Gutman) If $(X,T)$ is an extension of an aperiodic t.d.s which
has a countable number of minimal subsystems then it has the marker
property. \end{thm}
\begin{proof}
We may assume w.l.o.g that $(X,T)$ is aperiodic and has a countable
number of minimal subsystems. Let $n\in\mathbb{N}$. We will construct
inductively an open set $U\subset X$ so that the sets $\{T^{i}(U)\}_{i=1}^{n}$
are pairwise disjoint and $\{T^{i}(U)\}_{i=1}^{m}$ cover $X$ for
some $m$. Let $M_{1},M_{2},\ldots$ be an enumeration of the minimal
subsystems of $X$. Using the fact there is only a countable number
of minimal subsystems find $m_{1}\in M_{1}$, $r_{1}>0$ so that $\{T^{i}B_{r_{1}}(m_{1})\}_{i=-n}^{n}$
are pairwise disjoint and for all $l\geq2$, $M_{l}\nsubseteq\bigcup_{i=-n}^{n}T^{i}\partial B_{r_{1}}(m_{1})$
(here we use that $\{\bigcup_{i=-n}^{n}T^{i}\partial B_{r}(m_{1})\}_{r>0}$
is a uncountable collection of pairwise disjoint sets). Define $U_{1}=B_{r_{1}}(m_{1})$.
Assume one has defined an open set $U_{k}\subset X$ so that:
\begin{enumerate}
\item \label{enu:intersection with first subsystems}For any $i=1,\ldots,k$
there exists $j=j(i)\in\mathbb{Z}$ so that $U_{k}\cap T^{j}M_{i}\neq\emptyset$.
\item \label{enu:boundary condition}For all $l\geq k+1$, $M_{l}\nsubseteq\bigcup_{i=-n}^{n}T^{i}\partial U_{k}$.
\item \label{enu:-pairwise disjoint}$\{T^{i}(U_{k})\}_{i=1}^{n}$ are pairwise
disjoint
\end{enumerate}

\noindent If $U_{k}\cap T^{j}M_{k+1}\neq\emptyset$ for some $j\in\mathbb{Z}$,
define $U_{k+1}=U_{k}$. We now assume that $U_{k}\cap T^{j}M_{k+1}=\emptyset$
for all $j\in\mathbb{Z}$. By assumption $M_{k+1}\nsubseteq\bigcup_{i=-n}^{n}T^{i}\partial U_{k}$.
Conclude $M_{k+1}\nsubseteq\bigcup_{i=-n}^{n}T^{i}\overline{U}_{k}$.
Using the fact there is only a countable number of minimal subsystems,
we can find $m_{k+1}\in M_{k+1}$ and $r_{k+1}>0$ so that $\{T^{i}B_{r_{k+1}}(m_{k+1})\}_{i=-n}^{n}$
are pairwise disjoint and so that it holds:

\begin{equation}
B_{r_{k+1}}(m_{k+1})\cap\bigcup_{i=-n}^{n}T^{i}U_{k}=\emptyset,\label{eq:pairwise_disjoint}
\end{equation}

\begin{equation}
\forall l>k+1\quad M_{l}\setminus\bigcup_{i=-n}^{n}T^{i}\partial U_{k}\nsubseteq\bigcup_{i=-n}^{n}T^{i}\partial B_{r_{1}}(m_{1})\label{eq:k+1_Boundary_Condition}
\end{equation}

\noindent Define $U_{k+1}=U_{k}\cup B_{r_{k+1}}(m_{k+1})$. We now verify that
the desired properties hold:
\begin{enumerate}
\item For any $i=1,\ldots,k+1$ there exists $j=j(i)\in\mathbb{Z}$ so that
$U_{k+1}\cap T^{j}M_{i}\neq\emptyset$. Indeed if $i\leq k$, this
follows from property (\ref{enu:intersection with first subsystems})
above. For $i=k+1$ it is trivial.
\item For all $l\geq k+2$, $M_{l}\nsubseteq\bigcup_{i=-n}^{n}T^{i}\partial U_{k+1}$.
Indeed if follows from $\partial U_{k+1}\subset\partial U_{k}\cup\partial B_{r_{k+1}}(m_{k+1})$
and (\ref{eq:k+1_Boundary_Condition}).
\item $\{T^{i}(U_{k+1})\}_{i=1}^{n}$ are pairwise disjoint. Indeed it is
enough to show $T^{i_{1}}U_{k}\cap T^{i_{2}}B_{r_{k+1}}(m_{k+1})=\emptyset$
for all $1\leq i_{1},i_{2}\leq n$. This follows from (\ref{eq:pairwise_disjoint}).
\end{enumerate}

\noindent Finally we define $U=\bigcup_{k=1}^{\infty}U_{k}$. As $U_{1}\subset U_{2}\subset\cdots$,
it holds that $\{T^{i}(U)\}_{i=1}^{n}$ are pairwise disjoint. Clearly
for any $i\in\mathbb{N}$ there exists $j=j(i)\in\mathbb{Z}$ so that
$U\cap T^{j}M_{i}\neq\emptyset$. As $U$ is open and $M_{i}$ is
compact this implies there exists $m(i)\in\mathbb{N}$ so that $M_{i}\subset\bigcup_{l=0}^{m(i)}T^{l}U$.
Fix $x\in X$. There exists $i\in\mathbb{N}$ so that $\overline{orb(x)}\cap M_{i}\neq\emptyset$.
Conclude there exists $k\in\mathbb{Z}$, so that $T^{k}x\in\bigcup_{l=0}^{m(i)}T^{l}U$.
By a simple compactness argument we deduce the existence of $m\in\mathbb{N}$
so that$\{T^{i}(U)\}_{i=1}^{m}$ cover $X$.
\end{proof}

\begin{example}
Clearly the previous theorem applies to every t.d.s which consists
of a finite union of minimal systems. We now present a simple example
of an aperiodic t.d.s with an infinite countable number of minimal
systems. Let $C_{r}=\{(x,y)|\, x^{2}+y^{2}=r\}\subset\mathbb{R}^{2}$,
be a circle of radius $r$ around the origin. Select a strictly decreasing
sequence of positive numbers $r_{1}>r_{2}>\cdots$ with $r_{i}\rightarrow r_{0}>0$.
Let $X=\bigcup_{i=0}^{\infty}C_{r_{i}}\subset\mathbb{R}^{2}$ and
define $T:X\rightarrow X$ by rotating by $\alpha$ on each circle
where $\alpha$ is some irraitional number.
\end{example}

\begin{example}
\label{Ex:rotation times identity }Not every aperiodic system has
a countbale number of minimal subsystems. Indeed consider $X=\mathbb{T}^{2}$,
the two-dimensional torus equipped with $T(x,y)=(x+\alpha,y)$ for some $\alpha$
irraitional number.
\end{example}

\begin{defn}
Let $(X,T)$ be a t.d.s and denote by $\mathcal{M}$ the collection
of all minimal subspaces of $(X,T)$. $(X,T)$ has a \textbf{compact
minimal subsystems selector }if there exists a compact $L$ so that
for every $M\in\mathcal{M}$, $|L\cap M|=1$ and $L\subset\bigcup\mathcal{M}$. \end{defn}
\begin{thm}
\label{thm:compact minimal selector--->Marker Property}(Downarowicz)
If $(X,T)$ is an extension of an aperiodic t.d.s with a compact
minimal subsystems selector than it has the marker property.\end{thm}
\begin{proof}
We may assume w.l.o.g that $(X,T)$ has a compact minimal subsystems
selector $L$. Denote by $\mathcal{M}$ the collection of all minimal
subspaces of $(X,T)$. If $x,y\in L$, $x\neq y,$ then there exists
distinct $M_{x},M_{y}\in\mathcal{M}$ so that $x\in M_{x}$ and $y\in M_{y}.$
This implies $orb(x)\cap orb(y)=\emptyset$. As $(X,T$) is aperiodic
conclude the closed sets $\{T^{i}L\}_{i=-\infty}^{\infty}$ are pairwise
disjoint. Let $n\in\mathbb{N}$. There exists $\epsilon>0$ so that
$T^{i}B_{\epsilon}(L)$ $(1\leq i\leq n)$ are pairwise disjoint.
For every $M\in\mathcal{M}$ there exists by minimality $m=m(M)$
so that:

\[
M\subset\bigcup_{i=0}^{m}T^{i}B_{\epsilon}(L)
\]
\noindent For every $z\in X$, there exists $M_{z}\in\mathcal{M}$
so that $\overline{orb(z)}\cap M_{z}\neq\emptyset$. We conclude by
a compacness argument that $B_{\epsilon}(L)$, $T^{1}B_{\epsilon}(L)$, $T^{2}B_{\epsilon}(L),\ldots$,
eventually cover $X$.
\end{proof}

\begin{example}
A simple example of an aperiodic system with a compact minimal subsystems
selector is given by Example \ref{Ex:rotation times identity }. A
selector is given by ${\{0\}\times\mathbb{T}}$. An aperiodic system
with a compact minimal subsystems selector without a non trivial minimal
factor is given by taking a disjoint union of the previous exmaple
$X$ with a circle, $Y=X\stackrel{\circ}{\cup}\mathbb{T}$ where the
circle is equipped with a rotation by $\beta$, such that $\alpha$
and $\beta$ are incommensurable.
\end{example}

\begin{example}
Not all aperiodic t.d.s have a compact minimal subsystems selector.
Indeed let $X=\mathbb{T}^{2}$ be the two-dimensional torus and $T:X\rightarrow X$
be given by $T(x,y)=(x+\frac{1}{2},x+\alpha)$ for some $\alpha$
irrational. Let $\mathcal{M}$ the collection of all minimal subspaces
of $(X,T)$. Note $\mathcal{M}=\{\{t,t+\frac{1}{2}\}\times\mathbb{T}|\, t\in[0,\frac{1}{2})\}$.
Assume for a contradiction $(X,T)$ has a compact minimal selector
$L$. Let $L_{2}$ be the projection of $L$ on the first coordinate.
$L_{2}$ is a closed set so that $\mathbb{T}=L_{2}\stackrel{\circ}{\cup}(L_{2}+\frac{1}{2})$.
Contradiction.
\end{example}

\section{The Local Marker Property\label{sec:The-Local-Marker}}
\begin{defn}
\label{local marker property} Let $Z,W$ be closed sets with $Z\times W\subset(X\times X)\setminus(\triangle\cup(X\times P)\cup(P\times X))$. A subset $F$ of a t.d.s $(X,T)$ is called
a \textbf{local $n$-marker} ($n\in\mathbb{N}$) for $Z\times W$
if:\end{defn}
\begin{enumerate}
\item \label{enu:disjointness-1}$F\cap T^{i}(F)=\emptyset$ for $i=1,2,\ldots,n-1$.
\item \label{enu:covering-1}The sets $\{T^{i}(F)\}_{i=1}^{m}$, $i=0,1,\ldots, m-1$,
cover $Z\cup W$ for some $m$.
\end{enumerate}
We say $Z\times W$ has \textbf{local markers} if it has \textit{open}
$n$-markers for all $n\in\mathbb{N}$. We say that a cover of $Y\subset(X\times X)\setminus(\triangle\cup(X\times P)\cup(P\times X))$
by a countable collection of products of closed sets ${\{Z_{i}\times W_{i}\}_{i=1}^{\infty}}$
has the \textbf{local marker property relatively to $Y$} if for every
$i$, $Z_{i}\times W_{i}$ has local markers. We say $(X,T)$ has
the \textbf{local marker property} if there is a cover with the local
marker property relatively to $(X\times X)\setminus(\triangle\cup(X\times P)\cup(P\times X))$.
\begin{rem}
\label{Rem: Implications of local marker property}If $(X,T)$ has
the marker property than it has the local marker property.
\end{rem}

\begin{rem}
Just as in the case of the marker property, $Z\times W$ has local
markers iff it has closed $n$-markers for all $n\in\mathbb{N}$.

\end{rem}
\begin{thm}
\label{thm:finite dimensional and P closed implies local marker property}If
$\dim(X)<\infty$ and $P$ is closed, then $(X,T)$ has the local
marker property. \end{thm}
\begin{proof}
The proof follows closely the proof of Theorem 6.1 of \cite{Gut12a}
where it is shown that an aperiodic finite dimensional t.d.s has the
marker property. Theorem 6.1 of \cite{Gut12a} is based on a certain generalization of Lemma 3.7 of \cite{BC04} which is one of the building blocks in the proof of the Bonatti-Crovisier Tower Theorem for $C^1$-diffeomorphisms on manifolds \cite[Theorem 3.1]{BC04}. Let ${\{Z_{i}\times V_{i}\}_{i=1}^{\infty}}$ be
an arbitrary countable cover of $(X\times X)\setminus(\triangle\cup(X\times P)\cup(P\times X))$
by a countable collection of products of closed sets so that for every
$i$, $Z_{i}\cap P=\emptyset$ and $V_{i}\cap P=\emptyset$ (here
we use that $P$ is closed) and $Z_{i}\cap V_{i}=\emptyset$. Fix
$n,k\in\mathbb{N}$. For every $x\in Z_{k}\cup V_{k}$ choose an open
set $U_{x}$ so that $x\in U_{x}$, $\overline{U}_{x}\subset X\setminus P$
and $\overline{U}_{x}\cap T^{i}\overline{U}_{x}=\emptyset$ for $i=1,2,\ldots,m=(2dim(X)+2)n-1$.
Let $U_{x_{1}},U_{x_{2}},\ldots,U_{x_{s}}$ be a finite cover of $Z_{k}\cup V_{k}$.
We now continue exactly as in the proof of Theorem 6.1 of \cite{Gut12a},
to find a $W$, so that $\overline{W}\cap T^{i}\overline{W}=\emptyset,\, i=1,2,\ldots,n-1$
and $Z_{k}\cup V_{k}\subset\bigcup_{i=1}^{s}\overline{U}_{x_{i}}\subset\bigcup_{i=0}^{m}T^{i}(\overline{W})$
(we use the fact that $P$ is closed in order to invoke Lemma 6.2
of \cite{Gut12a}). The existence of an open $n$-marker for $Z_{k}\cup V_{k}$
follows easily.\end{proof}
\begin{prop}
\label{prop:If Omega has local marker property then X has the local marker property}If
$(\Omega(X),T)$ has the local marker property then $(X,T)$ has the
local marker property. \end{prop}
\begin{proof}
We will show there is a closed countable cover which has the local
marker property relatively to $S=(X\times X)\setminus(\triangle\cup(X\times P)\cup(P\times X))$
by defining $S_{1},S_{2},S_{3}\subset X\times X$ so that $S=S_{1}\cup S_{2}\cup S_{2}^{*}\cup S_{3}$
where $S_{2}^{*}={\{(y,x)|\,(x,y)\in S_{2}\}},$ and exhibiting $3$
closed countable covers which have the local marker property relatively
to $S_{1}$, $S_{2}$ and $S_{3}$ respectively (the case of $S_{2}^{*}$
will follow from the case of $S_{2}$). As $(\Omega(X),T)$ has the
local marker property, there exists a a countable cover $\{Z_{i}\times W_{i}\}_{i=1}^{\infty}$
which has the local marker property relatively to $(\Omega(X)\times\Omega(X))\setminus(\triangle\cup(\Omega(X)\times P)\cup(P\times\Omega(X)))$,
however it is important to note this is w.r.t the topology induced
by $\Omega(X)$. Also note that for all $i$, $Z_{i}\cup W_{i}\subset\Omega(X)$.
We now define $S_{1},S_{2},S_{3}$:
\begin{enumerate}
\item $S_{1}=(\Omega(X)\times\Omega(X))\cap S=(\Omega(X)\times\Omega(X))\setminus(\triangle\cup(\Omega(X)\times P)\cup(P\times\Omega(X))).$
We claim $\{Z_{i}\times W_{i}\}_{i=1}^{\infty}$ has the local marker
property relatively to $S_{1}$ (w.r.t the topology induced by $X$).
Indeed fix $k$. Let $F$ be a closed (in $\Omega(X)$ and therefore
in $X$) $n$-marker for $Z_{k}\times W_{k}$ in $\Omega(X)$. Clearly
one can find an $\epsilon>0$ so that $B_{\epsilon}(F)\subset X$
is an open $n$-marker for $Z_{k}\times W_{k}$ in $X$.
\item $S_{2}=((X\setminus\Omega(X))\times\Omega(X))\cap S$. As $X\times X$
is second-countable, every subspace is Lindelöf, i.e. every open cover
has a countable subcover. Let ${\{U_{i}\}_{i=1}^{\infty}}$ be an
open cover of $X\setminus\Omega(X)$ such that for each $i$ there
is an $\epsilon_{i}>0$ so that $\{T^{k}B_{\epsilon_{i}}(U_{i})\}_{k\in\mathbb{Z}}$
are pairwise disjoint. We claim the countable closed cover ${\{\overline{U}_{i}\times W_{k}\}_{i,k=1}^{\infty}}$
has the local marker property relatively to $S_{2}$. First observe
that as $B_{\epsilon_{i}}(U_{i})\subset X\setminus\Omega(X)$, $\overline{U}_{i}\cap W_{k}=\emptyset$.
Now fix $i,k,n$. Let $F\subset\Omega(X)$ be a closed $n$-marker
for $Z_{k}\times W_{k}$. Let $0<\delta<d(\bigcup_{j=-(n-1)}^{n-1}T^{j}B_{\epsilon_{i}/2}(U_{i}),\Omega(X))$
so that $B_{\delta}(F)$ is still an open $n$-marker for $Z_{k}\times W_{k}$.
We claim $B_{\delta}(F)\cup B_{\epsilon_{i}/2}(U_{i})$ is an open
$n$-marker for $\overline{U}_{i}\times W_{k}$. Indeed as $F\subset\Omega(X)$,
the choice of $\delta$ guarantees $T^{j_{1}}B_{\delta}(F)\cap T^{j_{2}}B_{\epsilon_{i/2}}(U_{i})=\emptyset$
for all $0\leq j_{1},j_{2}\leq n-1$.
\item $S_{3}=(X\setminus\Omega(X))^{2}\cap S$. Let ${\{U_{i}\}_{i=1}^{\infty}}$
be the open cover of the previous case. We can assume w.l.o.g that
for each $x,y\in(X\setminus\Omega(X))^{2}$ with $x\neq y$ there
are $i\neq k$ so that $x\in U_{i}$ and $y\in U_{k}$ and $\overline{U_{i}}\cap\overline{U_{k}}=\emptyset$.
Note $\mathcal{C}={\{\overline{U}_{i}\times\overline{U}_{k}|i,k\in\mathbb{Z},\:\overline{U_{i}}\cap\overline{U_{k}}=\emptyset\}}$
is a closed cover of $(X\setminus\Omega(X))^{2}\cap S$. A countable
closed cover which has the local marker property relatively to $S_{3}$,
will be achieved by splitting each member of the cover $\mathcal{C}$
to an union of at most a countable number of closed products. Fix
$i\neq k$ so that $\overline{U}_{i}\times\overline{U}_{k}\in\mathcal{C}$.
If for all $j\in\mathbb{Z}$, $T^{j}B_{\epsilon_{i}}(U_{i})\cap B_{\epsilon_{k}}(U_{k})=\emptyset$,
then $B_{\epsilon_{i}}(U_{i})\cup B_{\epsilon_{k}}(U_{k})$ is an
open $n$-marker of $\overline{U}_{i}\times\overline{U}_{k}$ for
all $n$. Assume this is not the case. Note that $\{\overline{U}_{i}\times(\overline{U}_{k}\cap T^{j}\overline{B}_{\epsilon_{i}/2}(U_{i}))\}_{j\in\mathbb{Z}}\cup\{\overline{U}_{i}\times(\overline{U}_{k}\setminus\bigcup_{j\in\mathbb{Z}}T^{j}B_{\epsilon_{i}/2}(U_{i}))\}$
is a countable closed cover of $\overline{U}_{i}\times\overline{U}_{k}$.
Let $j\in\mathbb{Z}$. Note that $B_{\epsilon_{i}}(U_{i})$ is an
open $n$-marker of $\overline{U}_{i}\times(\overline{U}_{k}\cap T^{j}\overline{B}_{\epsilon_{i}/2}(U_{i}))$
for all $n$. Fix $n$. As $d(\bigcup_{j=-(n-1)}^{n-1}\bigcup T^{k}(\overline{U}_{i}),\overline{U}_{k}\setminus\bigcup_{j\in\mathbb{Z}}T^{j}B_{\epsilon_{i}/2}(U_{i}))>0$,
there is $\delta>0$ so that $T^{j_{1}}B_{\delta}(\overline{U}_{i})\cap T^{j_{2}}B_{\delta}(\overline{U}_{k}\setminus\bigcup_{j\in\mathbb{Z}}T^{j}B_{\epsilon_{i}/2}(U_{i}))=\emptyset$
for all $0\leq j_{1},j_{2}\leq n-1$. Taking $\delta<min\{\epsilon_{i},\epsilon_{k}\}$,%
{} guarantees that $B_{\delta}(\overline{U}_{i})\cup B_{\delta}(\overline{U}_{k}\setminus\bigcup_{j\in\mathbb{Z}}T^{j}B_{\epsilon_{i}/2}(U_{i}))$
is an open $n$-marker for $\overline{U}_{i}\times\overline{U}_{k}\setminus\bigcup_{j\in\mathbb{Z}}T^{j}B_{\epsilon_{i}/2}(U_{i})$.
\end{enumerate}
\end{proof}
\section{The Strong Topological Rokhlin Property \label{sec:The-Strong-Topological}}

In \cite[Subsection 1.9]{G11} the topological Rokhlin property was
introduced (see also Subsection \ref{sub:Topological-Rokhlin-Property}).
Here is a stronger variant, originating in \cite{L99}:
\begin{defn}
We say that $(X,T)$ has the \textbf{(global) strong topological Rokhlin
property} if for every $n\in\mathbb{N}$ there exists a continuous
function $f:X\rightarrow\mathbb{R}$ so that if we define the \textit{exceptional
set} $E_{f}=\{x\in X\,|\, f(Tx)\neq f(x)+1\}$, then $T^{-i}(E_{f})$, $i=0,1,\ldots, n-1$, are pairwise disjoint.\end{defn}
\begin{rem}
\label{rem:at most one bad index (global case)}Assume $b-a\leq n-1$.
Under the above conditions consider $x\in X$. Then there exists at
most one index $a\leq l\leq b$ so that $f(T^{l+1}x)\neq f(T^{l}x)+1$.
Indeed if $T^{l}x,T^{l'}x\in E_{f}$ for $a\leq l<l'\leq b$, then
$E_{f}\cap T^{l-l'}E_{f}\neq\emptyset$ contradicting the definition.
\end{rem}

\begin{defn}
\label{def: local strong Rokhlin}We say that $(X,T)$ has the \textbf{local
strong topological Rokhlin property} if one can cover $(X\times X)\setminus(\triangle\cup(X\times P)\cup(P\times X))$
by a countable collection of products of closed sets ${\{Z_{i}\times W_{i}\}_{k=1}^{\infty}}$,
where for every $k$, $Z_{k}\cap W_{k}=\emptyset$ and for every $m\in\mathbb{N}$
and $a,b\in\mathbb{Z}$ with $a<b$ there exists a continuous function
$f:\bigcup_{i=a}^{b}T^{i}(Z_{k}\cup W_{k})\rightarrow\mathbb{R}$
so that if we define the \textit{exceptional set} $E_{f}=\{x\in\bigcup_{i=a}^{b-1}T^{i}(Z_{k}\cup W_{k})\,|\, f(Tx)\neq f(x)+1\}$,
then $T^{-i}(E_{f})$, $i=0,\ldots,m-1$, are pairwise disjoint
(note $x\in\bigcup_{i=a}^{b-1}T^{i}(Z_{k}\cup W_{k})$ implies that
both $f(x)$ and $f(Tx)$ are defined). In this context ${\{Z_{i}\times W_{i}\}_{k=1}^{\infty}}$
is also said to have the \textbf{local strong topological Rokhlin property}.\end{defn}
\begin{rem}
\label{rem:at most one bad index}Assume $m\geq b-a-2$. Under the
above conditions consider $x\in(Z_{k}\cup W_{k})$. Then there exists
at most one index $a\leq l\leq b-1$ so that $f(T^{l+1}x)\neq f(T^{l}x)+1$.
Indeed $T^{l}x\in\bigcup_{i=a}^{b-1}T^{i}(Z_{k}\cup W_{k})$ and $f(T^{l+1}x))\neq f(T^{l}x)+1$
imply $T^{l}x\in E_{f}$. If $T^{l}x,T^{l'}x\in E_{f}$ for $a\leq l<l'\leq b-1$,
then $E_{f}\cap T^{l-l'}E_{f}\neq\emptyset$ contradicting the definition.
\end{rem}
The following proposition is based on Lemma 3.3 of \cite{L99}, which is the statement that a system with an aperiodic minimal factor has the strong topological Rokhlin property:
\begin{prop}
\label{prop:local marker property-->local strong Rokhlin}If $(X,T)$
has the local marker property then $(X,T)$ has the local strong topological
Rokhlin property. \end{prop}
\begin{proof}
By assumption one can cover $(X\times X)\setminus(\triangle\cup(X\times P)\cup(P\times X))$
by a countable collection of products of closed sets ${\{Z_{i}\times W_{i}\}_{i=1}^{\infty}}$,
with the local marker property. We will now show that ${\{Z_{i}\times W_{i}\}_{i=1}^{\infty}}$
has the local strong topological Rokhlin property. Fix $m,k\in\mathbb{N}$
and $a,b\in\mathbb{Z}$ with $a<b$. Let $F$ be an open $m$ marker
for $Z_{k}\cup W_{k}$. It will be convenient to write $Z_{k}\cup W_{k}\subset\bigcup_{i=-a}^{q-b}T^{i}(F)$
for some $q>b-a$. Choose a closed $R\subset F$ so that $Z_{k}\cup W_{k}\subset\bigcup_{i=-a}^{q-b}T^{i}R$.
Conclude:

\begin{equation}
\bigcup_{i=a}^{b}T^{i}(Z_{k}\cup W_{k})\subset\bigcup_{i=0}^{q}T^{i}R\label{eq:R_covers}
\end{equation}
Let $\omega:X\rightarrow[0,1]$ be a continuous function so that $\omega_{|R}\equiv1$
and $\omega_{|F{}^{c}}\equiv0$. We define a random walk for $z\in X$.
At any point $p$ we arrive during the random walk, the walk terminates
with probability $\omega(p)$ and moves to $T^{-1}p$ with probability
$1-\omega(p)$. Notice that for every point $z\in\bigcup_{i=a}^{b}T{}^{i}(Z_k\cup W_k)$
the random walk will terminate after a finite number of steps. Indeed
by (\ref{eq:R_covers}) there is an $i\in\{0,\ldots,q\}$ so that
$z\in T^{i}R$, which implies the walk terminates in at most $q$
steps (when the point hits $R$). Conclude there is a finite number
of possible walks starting at $z$ and we denote by $f(z)$ the expected
length of the walk starting at $z$. As there is a uniform bound on
the length of walks, $f:\bigcup_{i=a}^{b}T^{i}(Z_{k}\cup W_{k})\rightarrow\mathbb{R}_{+}$
is continuous. Note that if $y\notin F$ and $y\in\bigcup_{i=a+1}^{b}T^{i}(Z_{k}\cup W_{k})$
, then $f(T^{-1}y)=f(y)-1$. Notice that $x\in\bigcup_{i=a}^{b-1}T^{i}(Z_{k}\cup W_{k})$
and $f(Tx)\neq f(x)+1$ implies $y\triangleq Tx\in\bigcup_{i=a+1}^{b}T^{i}(Z_{k}\cup W_{k})$
and $f(T^{-1}y)=f(TT^{-1}x)=f(x)\neq f(Tx)-1=f(y)-1$. Therefore in
such a case one must have $Tx=y\in F$. Conclude $E_{f}=\{x\in\bigcup_{i=a}^{b-1}T^{i}(\overline{U}\cup\overline{V})\,|\, f(Tx)\neq f(x)+1\}\subset T^{-1}F$.
We therefore have $T^{-i}(E_{f})\cap E_{f}=\emptyset\ i=1,\ldots,m-1$.

\end{proof}
The following question is interesting:
\begin{problem}
\label{Ques:Local Strong Rokhlin-->Local Marker?}Does the the local
strong topological Rokhlin property imply the local marker property?
\end{problem}
The question can be answered assuming the global strong topological
Rokhlin property:
\begin{thm}
\label{thm:Marker Property iff Strong Rokhlin}$(X,T)$ has
the strong topological Rokhlin property iff $(X,T)$ has the marker
property. \end{thm}
\begin{proof}
The fact that the marker property implies the strong topological Rokhlin
property follows by a similar argument to the proof of Proposition
\ref{prop:local marker property-->local strong Rokhlin}. To prove
the other direction assume $(X,T)$ has the strong topological Rokhlin
property. Fix $n\in\mathbb{N}$ and let $f:X\rightarrow\mathbb{R}$
be a continuous function such that for the open set $E_{f}=\{x\in X\,|\, f(Tx)\neq f(x)+1\}$,
$T^{-i}(E_{f})$, $i=0,1,\ldots, n-1$, are pairwise disjoint. We claim that
the iterates $E_{f}$,$T^{-1}E_{f},\ldots$ eventually cover $X$.
Indeed as $f$ is bounded from above for any $x\in X$, the series
$f(T^{i}x)$, $i=1,2,\ldots$ cannot increase indefinitely. %

\end{proof}

\section{An Embedding Theorem\label{sec:An-Embedding-Theorem}}

\subsection{The Baire Category Framework\label{sub:The-Baire-Catergory}}

We are interested in the question under which conditions a topological
dynamical system $(X,T)$ is embeddable in the \textit{$d$-cubical
shift} $(([0,1]^{d})^{\mathbb{Z}},\ shift)$ for some $d\in\mathbb{N}$.
Notice that a continuous function $f:X\rightarrow[0,1]^{d}$ induces
a continuous $\mathbb{Z}$-equivariant mapping $I_{f}:(X,T)\rightarrow(([0,1]^{d})^{\mathbb{Z}},shift)$
given by $x\mapsto(f(T^{k}x))_{k\in\mathbb{Z}}$, also known as the
\textit{orbit-map}. Conversely, any $\mathbb{Z}$-equivariant continuous
factor map $\pi:(X,T)\rightarrow(([0,1]^{d})^{\mathbb{Z}},shift)$
is induced in this way by $\pi_{0}:X\rightarrow[0,1]^{d}$, the projection
on the zeroth coordinate. We therefore study the space of continuous
functions $C(X,[0,1]^{d})$. Instead of explicitly constructing a
$f\in C(X,[0,1]^{d})$ so that $I_{f}:(X,T)\hookrightarrow(([0,1]^{d})^{\mathbb{Z}},shift)$
is an embedding, we show that the property of being an embedding $I_{f}:(X,T)\hookrightarrow(([0,1]^{d})^{\mathbb{Z}},shift)$
is \textit{generic} in $C(X,[0,1]^{d})$ (but without exhibiting an
explicit embedding). To make this precise introduce the following
definition:
\begin{defn}
Let $K\subset(X\times X)\setminus\Delta$ be a compact set and $f\in C(X,[0,1]^{d})$.
We say that $I_{f}$ is \textbf{$K$-compatible} if for every $(x,y)\in K$,
$I_{f}(x)\neq I_{f}(y)$, or equivalently if for every $(x,y)\in K$,
there exists $n\in\mathbb{Z}$ so that $f(T^{n}x)\neq f(T^{n}y)$.
Define:
\[
D_{K}=\{f\in C(X,[0,1]^{d})|\, I_{f}\,\,\mathrm{is\,\, K}-\mathrm{compatible}\}
\]
By Lemma A.2 of \cite{Gut12a}, $D_{K}$ is open in $C(X,[0,1]^{d})$.
Under the assumptions of the theorem below we will show there exists
a closed countable cover $\mathcal{K}$ of $(X\times X)\setminus\Delta$
so that $D_{K}$ is dense for all $K\in\mathcal{K}$. By the Baire
Category Theorem $(C(X,[0,1]^{d}),||\cdot||_{\infty})$, is a \textit{Baire
space}, i.e., a topological space where the intersection of countably
many dense open sets is dense. This implies $\bigcap_{K\in\mathcal{K}}D_{K}$
is dense in $(C(X,[0,1]^{d}),||\cdot||_{\infty})$. Any $f\in\bigcap_{K\in\mathcal{K}}D_{K}$
is $K$-compatible for all $K\in\mathcal{K}$ simultaneously and therefore
induces an embedding $I_{f}:(X,T)\hookrightarrow(([0,1]^{d})^{\mathbb{Z}},shift)$.
A set in a topological space is said to be \textit{comeagre} or \textit{generic}
if it is the complement of a countable union of nowhere dense sets.
A set is said to be $G_{\delta}$ if it is the countable intersection
of open sets. As a dense $G_{\delta}$ set is comeagre, the above
argument shows that the set $\mathcal{A}\subset C(X,[0,1]^{d})$ for
which $I_{f}:(X,T)\hookrightarrow(([0,1]^{d})^{\mathbb{Z}},shift)$
is an embedding is comeagre, or equivalently, that the fact of $I_{f}$
being an embedding is generic in $(C(X,[0,1]^{d}),||\cdot||_{\infty})$.
\end{defn}

\subsection{The Embedding Theorem}
\begin{thm}
\label{thm:Local Marker + P-Embedding implies 36 Embedding}Assume
$(X,T)$ has the local marker property. Let $d\in\mathbb{N}$ be such
that $mdim(X,T)<\frac{d}{36}$ and $perdim(X,T)<\frac{d}{2}$, then
the collection of continuous functions $f:X\rightarrow[0,1]^{d}$
so that $I_{f}:(X,T)\hookrightarrow(([0,1]^{d})^{\mathbb{Z}},shift)$
is an embedding is comeagre in $C(X,[0,1]^{n})$.\end{thm}
\begin{proof}
As explained in Subsection \ref{sub:The-Baire-Catergory} we need
to exhibit a closed countable cover $\mathcal{C}$ of $(X\times X)\setminus\Delta$
so that $D_{C}$ is dense for all $C\in\mathcal{C}$. By Proposition
\ref{prop:local marker property-->local strong Rokhlin} $(X,T)$
has the local strong topological Rokhlin property. By Proposition
\ref{lem:lstrp-->D_K dense} below one can cover $X\times X\setminus\big(\triangle\cup(X\times P)\cup(P\times X)\big)$
by a closed countable cover $\mathcal{W}$ so that for all $W\in\mathcal{W}$,
$D_{W}$ is dense in $C(X,[0,1]^{d})$. Let $P_{n}$ denote the set
of points of period $\leq n$ and define $H_{n}=P_{n}\setminus P_{n-1}$.
By Proposition \ref{prop:D_K is dense (periodic case)} below for
every $n\in\mathbb{N}$ there is a countable closed cover $\mathcal{K}_{n}$
of $\big((X\setminus P)\times H_{n}\big)\cup\big(H_{n}\times(X\setminus P)\big)$
so that for all $K\in\mathcal{K}_{n}$, $D_{K}$ is dense in $C(X,[0,1]^{d})$.
By the proof of Theorem 4.1 of \cite{Gut12a} there is closed countable
cover $\mathcal{P}$ of $(P\times P)\setminus\Delta$ so that $D_{K}$
is dense for all $K\in\mathcal{K}$. Let $\mathcal{C}=\mathcal{W}\cup\mathcal{P}\cup\bigcup_{n}\mathcal{K}_{n}$.
As $(X\times X)\setminus\Delta$ is the union of $(X\times X)\setminus\big(\triangle\cup(X\times P)\cup(P\times X)\big)$,
$\bigcup_{n}\big((X\setminus P)\times H_{n}\big)\cup\big(H_{n}\times(X\setminus P)\big)=\big((X\setminus P)\times P\big)\cup\big(P\times(X\setminus P)\big)$
and $(P\times P)\setminus\Delta$, clearly $\mathcal{C}$ has the
desired properties. We now proceed to prove Proposition \ref{lem:lstrp-->D_K dense}
and Proposition \ref{prop:D_K is dense (periodic case)}. Throughout,
it turns out to be convenient to define:

\[
m_{dim}=\begin{cases}
\frac{1}{72} & \quad mdim(X,T)=0\\
mdim(X,T) & \quad\mathrm{otherwise}
\end{cases}
\]
Notice it holds $36m_{dim}<d$.\end{proof}
\begin{prop}
\label{lem:lstrp-->D_K dense}Let $\mathcal{K}$ be a countable closed
cover of $X\times X\setminus\big(\triangle\cup(X\times P)\cup(P\times X)\big)$
which has the local marker property, then for all $K\in\mathcal{K}$,
$D_{K}$ is dense in $C(X,[0,1]^{d})$.\end{prop}
\begin{proof}
This proof is heavily influenced by the proof of Theorem 5.1 of \cite{L99}.
Fix $K\in\mathcal{K}$ with $K=Z\times W$ with $Z,W$ closed and
$Z\cap W=\emptyset$. Fix $\epsilon>0$. Let $\tilde{f}:X\rightarrow[0,1]^{d}$
be a continuous function. We will show that there exists a continuous
function $f:X\rightarrow[0,1]^{d}$ so that $\|f-\tilde{f}\|_{\infty}<\epsilon$
and $I_{f}$ is $K$-compatible. We start by a general construction
and then relate it to $\tilde{f}$. Let $\delta=dist(Z,W)>0$. Let
$\alpha$ be a cover of $X$ with $\max_{U\in\alpha}diam(\tilde{f}(U))<\frac{\epsilon}{2}$
and $\max_{U\in\alpha}diam(U)<\delta$. Let $\epsilon'>0$ be such
that $36m_{dim}(1+2\epsilon')<d$. Let $N\in\N$ be such that it holds
$\frac{1}{N}D(\alpha^{N+1})<(1+\epsilon')m_{dim}$ (here we use $m_{dim}>0$
and Remark 2.5.1 of \cite{Gut12a}), $36(1+2\epsilon')<\epsilon'N$
and $N$ is divisible by $36$. Let $\gamma\succ\alpha^{N+1}$ be
an open cover so that $D(\alpha^{N+1})=ord(\gamma)$. We have thus
$ord(\gamma)<N(1+\epsilon')m_{dim}.$ Let $M=\frac{2}{9}N$ and $\Delta=\frac{M}{8}-1$.
Notice $M,\Delta\in\mathbb{N}$. Notice $\Delta d>(\frac{N}{36}-1)36m_{dim}(1+2\epsilon')=Nm_{dim}(1+\epsilon')+m_{dim}(N\epsilon'-36(1+2\epsilon'))>Nm_{dim}(1+\epsilon')$.
Conclude:

\[
ord(\gamma)<\Delta d
\]
 For each $U\in\gamma$ choose $q_{U}\in U$ so that $\{q_{U}\}_{U\in\gamma}$
is a collection of distinct points in $X$, and define $\tilde{v}_{U}=(\tilde{f}(T^{i}q_{U}))_{i=0}^{N-1}$.
According to Lemma 5.6 of \cite{L99}, one can find a continuous function
$F:X\rightarrow([0,1]^{d})^{N}$, with the following properties:
\begin{enumerate}
\item $\forall U\in\gamma$, $||F(q_{U})-\tilde{v}_{U}||_{\infty}<\frac{\epsilon}{2}$,
\item $\forall x\in X$, $F(x)\in co\{F(q_{U})|\, x\in U\in\gamma\}$,
\item \label{enu:linear-independence}If for some $0\leq l,j<N-4\Delta$
and $\lambda,\lambda'\in(0,1]$ and $x,y,x',y'\in X$ so that:
\end{enumerate}

\[
\lambda F(x)|_{l}^{l+4\Delta-1}+(1-\lambda)F(y)|_{l+1}^{l+4\Delta}=\lambda'F(x')|_{j}^{j+4\Delta-1}+(1-\lambda')F(y')|_{j+1}^{j+4\Delta}
\]
then there exist $U\in\gamma$ so that $x,x'\in U$ and $l=j$ (note the statement $l=j$ is missing from Lemma 5.6 of \cite{L99} but follows from the proof).

By Proposition \ref{prop:local marker property-->local strong Rokhlin},
$(X,T)$ has the local strong topological Rokhlin property. By Definition
\ref{def: local strong Rokhlin} one can find a continuous function
$n:\bigcup_{-4M}^{\frac{M}{2}-1}T^{i}(Z\cup W)\rightarrow\mathbb{R}$
so that for $E_{n}=\{x\in\bigcup_{-4M}^{\frac{M}{2}-2}T^{i}(Z\cup W)\,|\, n(Tx)\neq n(x)+1\}$
one has $E_{n}\cap T^{i}(E_{n})=\emptyset$ for $1\leq i\leq\frac{9}{2}M-1$.
Let $\underline{n}(x)=\lfloor n(x)\rfloor\,\mod M$, $\overline{n}(x)=\lceil n(x)\rceil\,\mod M$,
$n'(x)=\{n(x)\}$. Let $A=\bigcup_{-4M}^{\frac{M}{2}-1}T^{i}(Z\cup W)$.
Define:

\begin{equation}
f'(x)=(1-n'(x)F(T^{-\underline{n}(x)}x)|_{\underline{n}(x)}+n'(x)F(T^{-\overline{n}(x)}x)|_{\overline{n}(x)}\hspace{1em}x\in A\label{eq:f'}
\end{equation}
$f'$ is continuous by the argument appearing on p. 241 of \cite{L99}.
By the argument of Claim 1 on p. 241 of \cite{L99}, as $\max_{U\in\alpha}diam(\tilde{f}(U))<\frac{\epsilon}{2}$
and $\max_{U\in\gamma}||F(q_{U})-v_{U}||_{\infty}<\frac{\epsilon}{2}$
we have $||\tilde{f}_{|A}-f'||_{A,\infty}<\epsilon$. By Lemma A.5
of \cite{Gut12a} there is $f:X\rightarrow[0,1]^{d}$ so that $f{}_{|A}=f'_{|A}$
and $||f-\tilde{f}||_{\infty}<\epsilon$. We now show that $f\in D_{K}$.
Fix $x'\in Z$ and $y'\in W$. Assume for a contradiction $f(T^{a}x')=f(T^{a}y')$
for all $a\in\mathbb{Z}$. Notice that by Remark \ref{rem:at most one bad index}
for both $x',y'$ there is at most one index $-4M\leq j_{x'},j_{y'}\leq\frac{M}{2}-2$
for which $n(T^{j{}_{x'}+1}x')\neq n(T^{j{}_{x'}}x')+1$, $n(T^{j{}_{y'}+1}y')\neq n(T^{j{}_{y'}}y')+1$
respectively. By Lemma \ref{lem:Existence_simultaneous_good_segment}
one can find an index $-4M\leq r\leq0$ so that for $r\leq s\leq r+\frac{M}{2}-2$,
for $z'=x',y'$, $\underline{n}(T^{s}z')=\underline{n}(T^{r}z')+s-r$,
$\overline{n}(T^{s}z')=\overline{n}(T^{r}z')+s-r$ and $\underline{n}(T^{r}z')\leq\frac{M}{2}$.
Denote $\lambda=n'(T^{r}x')$, $\lambda'=n'(T^{r}y')$, $a=\underline{n}(T^{r}x')\leq\frac{M}{2}$
and $a'=\underline{n}(T^{r}y')$. Substituting $T^{s}x',T^{s}y'$
for $r\leq s\leq r+4\Delta-1=r+\frac{M}{2}-5$ in equation (\ref{eq:f'})
(note $T^{s}x',T^{s}y'\in A$), we conclude from the equality $I_{f}(x')|_{r}^{r+\frac{M}{2}-5}=I_{f}(y')|_{r}^{r+\frac{M}{2}-5}$:

\[
(1-\lambda)F(T^{r-a}x')|_{a}^{a+4\Delta-1}+\lambda F(T^{r-a-1}x')|_{a+1}^{a+4\Delta}=(1-\lambda')F(T^{r-a'}y')|_{a'}^{a'+4\Delta-1}+\lambda'F(T^{r-a'-1}y')|_{a'+1}^{a'+4\Delta}
\]
E.g. notice that for $0\leq i\leq\frac{M}{2}-5$ it holds that $T^{-\underline{n}(T^{r+i}x')}T^{r+i}x'=T^{-(a+i)+r+i}x'=T^{r-a}x'$.
As the conditions of Lemma 5.6 of \cite{L99} are fulfilled then
by condition (\ref{enu:linear-independence}), one has that $a=a'$
and that there exist $U\in\gamma\succ\alpha^{N+1}$ so that $T^{r-a}x',T^{r-a}y'\in U$.
As $N=-4M-\frac{M}{2}\leq r-a\leq0$ we can find $V\in\alpha$, so
that $x',y'\in V$. This is a contradiction to $\max_{U\in\alpha}diam(U)<dist(Z,W)=\delta$.

\end{proof}
\begin{prop}
\label{prop:D_K is dense (periodic case)}Assume $(X,T)$ has the
local strong topological Rokhlin property and let $n\in\mathbb{N}$,
then there is a countable closed cover $\mathcal{K}$ of $(X\setminus P)\times H_{n}$
so that for $K\in\mathcal{K}$, $D_{K}$ is dense in $C(X,[0,1]^{d})$.\end{prop}
\begin{proof}
Let $\mathcal{C}$ be a cover of $(X\times X)\setminus(\triangle\cup(X\times P)\cup(P\times X))$
with the local strong topological Rokhlin property. Cover $H_{n}$
by a countable collection. Let $W$ be an open set in $H_{n}$ (not
necessarily open in $X$) with $y\in W\subset\overline{W}\subset H_{n}$.
Let $Z\times R\in\mathcal{C}$. Fix $\epsilon>0$. Let $\tilde{f}:X\rightarrow[0,1]^{d}$
be a continuous function. We will show that there exists a continuous
function $f:X\rightarrow[0,1]^{d}$ so that $\|f-\tilde{f}\|_{\infty}<\epsilon$
and $I_{f}$ is $K$-compatible for $K=Z\times\overline{W}$. Let
$\alpha$ be a cover of $X$ with $\max_{U\in\alpha}diam(\tilde{f}(U))<\min\{\frac{\epsilon}{2},d(Z,P_{n})\}$.
Let $\epsilon'>0$ be such that $36m_{dim}(1+2\epsilon')<d$. We will
see it is enough to assume $18m_{dim}(1+2\epsilon')<d$ (actually
it is enough to assume $8m_{dim}(1+2\epsilon')<d$ but we will not
use this fact). Let $N\in\N$, divisible by $18$, be such that it
holds $\frac{1}{N}D(\alpha^{N})<(1+\epsilon')m_{dim}$ and $N\epsilon'-9n(1+2\epsilon')>\frac{1}{m_{dim}}$.
Let $\gamma\succ\alpha^{N}$ be an open cover so that $D(\alpha^{N})=ord(\gamma)$.
Let $M=\frac{2}{9}N$ and $S=\frac{M}{4}$. Notice $(S-\frac{n}{2})d>(\frac{N}{18}-\frac{n}{2})18m_{dim}(1+2\epsilon')=Nm_{dim}(1+\epsilon')+m_{dim}(N\epsilon'-9n(1+2\epsilon'))>1+Nm_{dim}(1+\epsilon')$. As $ord(\gamma)<N(1+\epsilon')m_{dim}$, conclude:

\[
ord(\gamma)+1<(S-\frac{n}{2})d
\]
For each $U\in\gamma$ choose $q_{U}\in U$ so that $\{q_{U}\}_{U\in\gamma}$
is a collection of distinct points in $X$, and define $\tilde{v}_{U}=(\tilde{f}(T^{i}q_{U}))_{i=0}^{N-1}$.
According to Lemma \ref{lem:F not in V_n} one can find a continuous
function $F:X\rightarrow([0,1]^{d})^{N}$, with the following properties:
\begin{enumerate}
\item $\forall U\in\gamma$, $||F(q_{U})-\tilde{v}_{U}||_{\infty}<\frac{\epsilon}{2}$,
\item $\forall x\in X$, $F(x)\in co\{F(q_{U})|\, x\in U\in\gamma\}$,
\item \label{enu:not in V_n}For any $0\leq l,j<N-2S$, and $\lambda\in(0,1]$
and $x,y\in X$ with $d(x,y)>\max_{W\in\gamma}diam(W)$ it holds:
\end{enumerate}

\[
(1-\lambda)F(x)|_{l}^{l+2S-1}+\lambda F(y)|_{l+1}^{l+2S}\notin V_{2S}^{n}
\]
where,

\[
V_{2S}^{n}\triangleq\{w=(w_{0},\ldots,w_{2\Delta-1})\in([0,1]^{d})^{2\Delta}|\forall0\leq a,b\leq2S-1,a=b\,\mod n\rightarrow w_{a}=w_{b}\}.
\]

\noindent By Definition \ref{def: local strong Rokhlin} one can
find a continuous function $n:\bigcup_{-4M}^{\frac{M}{2}-1}T^{i}Z\rightarrow\mathbb{R}$
so that for $E_{n}=\{x\in\bigcup_{-4M}^{\frac{M}{2}-2}T^{i}Z\,|\, n(Tx)\neq n(x)+1\}$
one has $E_{n}\cap T^{i}(E_{n})=\emptyset$ for $1\leq i\leq\frac{9}{2}M-1$.
Let $\underline{n}(x)=\lfloor n(x)\rfloor\,\mod M$, $\overline{n}(x)=\lceil n(x)\rceil\,\mod M$,
$n'(x)=\{n(x)\}$. Let $A=\bigcup_{-4M}^{\frac{M}{2}-1}T^{i}Z$. Define:

\begin{equation}
f'(x)=(1-n'(x)F(T^{-\underline{n}(x)}x)|_{\underline{n}(x)}+n'(x)F(T^{-\overline{n}(x)}x)|_{\overline{n}(x)}\hspace{1em}x\in A\label{eq:f'-2}
\end{equation}
$f'$ is continuous by the argument appearing on p. 241 of \cite{L99}.
By the argument of Claim 1 on p. 241 of \cite{L99}, as $\max_{U\in\alpha}diam(\tilde{f}(U))<\frac{\epsilon}{2}$
and $\max_{U\in\gamma}||F(q_{U})-\tilde{v}_{U}||_{\infty}<\frac{\epsilon}{2}$
we have $||f'-\tilde{f}_{|A}||_{\infty}<\epsilon$. By Lemma A.5 of
\cite{Gut12a} there is $f:X\rightarrow[0,1]^{d}$ so that $f{}_{|A}=f'_{|A}$
and $||f-\tilde{f}||_{\infty}<\epsilon$. We now show that $f\in D_{K}$.
Fix $x'\in Z$ and $y'\in\overline{W}$. Assume for a contradiction
$f(T^{a}x')=f(T^{a}y')$ for all $a\in\mathbb{Z}$. Notice that by
Remark \ref{rem:at most one bad index} there is at most one index
$-4M\leq j_{x'}\leq\frac{M}{2}-2$ for which $n(T^{j{}_{x'}+1}x')\neq n(T^{j{}_{x'}}x')+1$.
By Lemma \ref{lem:Existence_simultaneous_good_segment} one can find
an index $-4M\leq r\leq0$ so that for $r\leq s\leq r+\frac{M}{2}-1$,
$\underline{n}(T^{s}x')=\underline{n}(T^{r}x')+s-r$ and $\overline{n}(T^{s}x')=\overline{n}(T^{r}x')+s-r$.
Denote $\lambda=n'(T^{r}x')$ and $a=(T^{r}x')$. Substituting $T^{s}x'$
for $r\leq s\leq r+2S-1=r+\frac{M}{2}-1$ in equation (\ref{eq:f'-2})
(note $T^{s}x'\in A$), we conclude from the equality $I_{f}(x')|_{r}^{r+\frac{M}{2}-1}=I_{f}(y')|_{r}^{r+\frac{M}{2}-1}$
(compare with the analogue part in the proof of Proposition \ref{lem:lstrp-->D_K dense}):

\[
(1-\lambda)F(T^{r-a}x')|_{a}^{a+2S-1}+\lambda F(T^{r-a-1}x')|_{a+1}^{a+2S}=I_{f}(y')|_{r}^{r+2S-1}
\]
As $y'\in\overline{W}\subset H_{n}$, one clearly has $I_{f}(y')|_{r}^{r+2S-1}\in V_{2S}^{n}$.
This is a contradiction to property (\ref{enu:not in V_n}).

\end{proof}

\section{Applications\label{sec:Applications}}
\begin{thm}
Assume $(X,T)$ is an extension of an aperiodic t.d.s which either
is finite-dimensional or has a countable number of minimal subsystems or has a compact minimal subsystems selector. Then $(X,T)$ has the strong Rokhlin property. If in addition $d\in\mathbb{N}$
is such that $mdim(X,T)<\frac{d}{36}$, then the collection of continuous
functions $f:X\rightarrow[0,1]^{d}$ so that $I_{f}:(X,T)\hookrightarrow(([0,1]^{d})^{\mathbb{Z}},shift)$
is an embedding is comeagre in $C(X,[0,1]^{d})$.\end{thm}
\begin{proof}
In those cases, by Theorems \ref{thm:countable number of minimal systems-->Marker Property},
\ref{thm:compact minimal selector--->Marker Property} as well as
Theorem 6.1 of \cite{Gut12a}, $(X,T)$ has the marker property. We
can therefore conclude by Theorem \ref{thm:Marker Property iff Strong Rokhlin}
that $(X,T)$ has the strong Rokhlin property. By Theorem \ref{thm:Local Marker + P-Embedding implies 36 Embedding}
, as $(X,T)$ is aperiodic the second part of the theorem holds. \end{proof}
\begin{lem}
\label{lem:mdim(X,T)=00003Dmdim(Omega(X),T)}$mdim(X,T)=mdim(\Omega(X),T)$\end{lem}
\begin{proof}
Clearly $mdim(X,T)\geq mdim(\Omega(X),T)$. To see the reversed inequality
let $Y=X/\Omega(X)$ (i.e. the quotient space where the closed and
$T$-invariant subspace $\Omega(X)$ is identified with a point) and
let $\pi:(X,T)\rightarrow(Y,T')$ be the quotient map, where $T'$
is the induced transformation. Note that $\pi_{X\setminus\Omega(X)}$
is injective. We can therefore use Theorem 4.6 of \cite{Tsu08} in
order to conclude $mdim(X,T)\leq mdim(Y,T')+mdim(\Omega(X),T)$. As
$\pi(\Omega(X))\simeq\{\bullet\}$ is the only closed invariant subsystem
of $Y$, it holds $h_{top}(Y,T')=0$ which implies $mdim(Y,T')=0$
(see Subsection \ref{sub:The-Metric-Mean}). We therefore conclude
$mdim(X,T)\leq mdim(\Omega(X),T)$ as desired. \end{proof}
\begin{thm}
\label{thm:Omega finite dimensional-->Embedding}Let $(X,T)$ be a
t.d.s so that $\Omega(X)$ is finite dimensional, the set of periodic points $P(X,T)$ is closed
and $perdim(X,T)<\frac{d}{2}$. Then the collection of continuous
functions $f:X\rightarrow[0,1]^{d}$ so that $I_{f}:(X,T)\hookrightarrow(([0,1]^{d})^{\mathbb{Z}},shift)$
is an embedding is comeagre in $C(X,[0,1]^{d})$.\end{thm}
\begin{proof}
By Lemma \ref{lem:mdim(X,T)=00003Dmdim(Omega(X),T)} as $\Omega(X)$
is finite dimensional, $mdim(X,T)=mdim(\Omega(X),T)=0$. By Theorem
\ref{thm:finite dimensional and P closed implies local marker property}
$(\Omega(X),T)$ has the local marker property. By Proposition \ref{prop:If Omega has local marker property then X has the local marker property}
$(X,T)$ has the local marker property. Combining all of these facts,
we conclude by Theorem \ref{thm:Local Marker + P-Embedding implies 36 Embedding}
that the collection of continuous functions $f:X\rightarrow[0,1]^{d}$
so that $I_{f}:(X,T)\hookrightarrow(([0,1]^{d})^{\mathbb{Z}},shift)$
is an embedding is comeagre in $C(X,[0,1]^{d})$.
\end{proof}
Recall the Lindenstrauss-Tsukamoto Conjecture from the Introduction.
\begin{cor}
Let $(X,T)$ be a t.d.s so that $\Omega(X)$ is finite dimensional
and the set of periodic points $P(X,T)$ is closed, then the Lindenstrauss-Tsukamoto Conjecture
holds for $(X,T)$.\end{cor}
\begin{proof}
It is sufficient to notice that $mdim(X,T)=0$ (as pointed out in
the proof of Theorem \ref{thm:Omega finite dimensional-->Embedding})
and apply Theorem \ref{thm:Omega finite dimensional-->Embedding}.
\end{proof}

\begin{example}
We now construct a family of examples for which the previous theorem is
applicable. Let $R:[0,1]\rightarrow[0,1]$ be a continuous invertible
map such that $R(0)=0$, $R(1)=1$ and such there are no other fixed points. It easily follows that for all $0<x<1$, $\lim_{n\rightarrow\infty}R^{n}(x)=1$,
$\lim_{n\rightarrow-\infty}R^{n}(x)=0$ or $\lim_{n\rightarrow\infty}R^{n}(x)=0$,
$\lim_{n\rightarrow-\infty}R^{n}(x)=1$, e.g. $R(x)=\sqrt{x},x^2$. Let
$Q=[0,1]^{\mathbb{N}}$, be the Hilbert cube, equipped with the product
topology. Define $\mathbf{R}:Q\rightarrow Q$, by $\mathbf{R}((x_{i})_{i=1}^{\infty})~=~(R(x_{i}))_{i=1}^{\infty}$.
It is easy to see $\Omega(Q,\mathbf{R})={\{0,1\}^{\mathbb{N}}}$.
Let $(Y,S)$ be a finite dimensional t.d.s with a closed set of periodic
points. It follows easily that $\Omega(Y\times Q,S\times\mathbf{R})=\Omega(Y,S)\times\Omega(Q,\mathbf{R})=\Omega(Y,S)\times{\{0,1\}^{\mathbb{N}}}$.
As $\{0,1\}^{\mathbb{N}}$ is zero-dimensional and $\Omega(Y,S)\subset Y$,
we conclude $\Omega(Y\times Q,S\times\mathbf{R})$ is finite-dimensional.
Moreover $P(Y\times Q,S\times\mathbf{R})=P(Y,S)\times{\{0,1\}^{\mathbb{N}}}$,
which is closed. We have thus verified all prerequisites that enable
us to apply the previous theorem for the infinite-dimensional system $(Y\times Q,S\times\mathbf{R})$.
Additionally notice that as $\{0,1\}^{\mathbb{N}}$ consists of fixed
points of $(Q,\mathbf{R})$ and is zero-dimensional, $\overrightarrow{perdim}(Y\times Q,S\times\mathbf{R})=\overrightarrow{perdim}(Y,S)$. \end{example}

\section{The Equivalence of SBP and Vanishing Mean Dimension Under the Marker
Property\label{sec:Equivalence-SBP-Mdim=00003D0}}

Recall the definition of $mdim_{d}(X,T)$ in Subsection \ref{sub:The-Metric-Mean}.
\begin{thm}
\label{thm:Metric Mean Diemension}If $(X,T)$ has the marker property
then there is a compatible metric $d'$ such that $mdim(X,T)=mdim_{d'}(X,T)$.\end{thm}
\begin{proof}
This is a straightforward generalization of Theorem 4.3 of \cite{L99},
which is the statement that the conclusion of the theorem holds if
the system has an aperiodic minimal factor.
\end{proof}
As a corollary of the previous theorem we have the following theorem:
\begin{thm}
Assume $(X,T)$ is an extension of an aperiodic t.d.s which either
is finite dimensional or has a countable number of minimal subsystems
or has a compact minimal subsystems selector then there is
a compatible metric $d'$ such that $mdim(X,T)=mdim_{d'}(X,T)$.
\end{thm}

\begin{thm}
\label{thm:SBP iff Mdim=00003D0} If $(X,T)$ has the marker property
then the following conditions are equivalent:

(a) $mdim(X,T)=0$

(b) $(X,T)$ has the small boundary property (SBP)

(c)$(X,T)=\lim_{\substack{\longleftarrow\\
i\in\mathbb{N}
}
}(X_{i},T_{i})$ where $h_{top}(X_{i},T_{i})<\infty$ for $i\in\mathbb{N}$.\end{thm}
\begin{proof}
$(a)\Rightarrow(b)$ is straightforward generalization of Theorem
6.2 of \cite{L99}, which is the statement that $(a)\Rightarrow(b)$
holds if the system has an aperiodic minimal factor. $(c)\Rightarrow(a)$
follows from Proposition 2.8 of \cite{LW} (this implication is true
for any system). $(b)\Rightarrow(a)$ is Theorem 5.4 of \cite{LW}
(this implication is true for any system). $(a)\Rightarrow(c)$ is
straightforward generalization of Proposition 6.14 of \cite{L99},
which is the statement that $(a)\Leftrightarrow(c)$ holds if the
system has an aperiodic minimal factor.\end{proof}
As a corollary of the previous theorem we have the following theorem:
\begin{thm}
Assume $(X,T)$ is an extension of an aperiodic t.d.s which either
is finite dimensional or has a countable number of minimal subsystems
or has a compact minimal subsystems selector then then $mdim(X,T)=0$
iff $(X,T)$ has the small boundary property iff $(X,T)=\lim_{\substack{\longleftarrow\\
i\in\mathbb{N}
}
}(X_{i},T_{i})$ where $h_{top}(X_{i},T_{i})<\infty$ for $i\in\mathbb{N}$.
\end{thm}

\section{A characterization of the small boundary property through isomorphic
extensions\label{sec:A-characterization-of SBP}}

The following section answers a question of Tomasz Downarowicz in
the affirmative (see Theorem \ref{thm:Affirmative Answer}). In this
section (and only in this section) a \textbf{topological dynamical
system} (t.d.s) $(X,T)$ consists of a compact metric space $X$ and
a continuous transformation (not necessarily invertible) $T~:~X\rightarrow X$.
We denote by $M_{T}(X)$ the set of $T$-invariant Borel (probability)
measures and by $\mathcal{B}_{\mu}$ the $\sigma$-algebra of Borel
sets completed with respect to $\mu$. A \textbf{measure preserving
system} (m.p.s) is a quadratuple $(Y,\mathcal{C},\nu,T)$ where $(Y,\mathcal{C},\nu)$
is a probability space and $T:(Y,\mathcal{C},\nu)\rightarrow(Y,\mathcal{C},\nu)$
is a measure preserving transformation.
\begin{defn}
$\pi:(Z,S)\rightarrow(X,T)$ is an \textbf{isomorphic
extension} iff for any $\mu\in M_{S}(Z)$, $\pi:(Z,S,\mathcal{B}_{\mu},\mu)\rightarrow(X,T,\mathcal{B}_{\pi_{\ast}\mu},\pi_{\ast}\mu)$
is an m.p.s isomorphism.
\end{defn}
The following trivial proposition is included for completeness.
\begin{prop}
If D is closed then $\partial\stackrel{\circ}{D}\subset\partial D$.\end{prop}
\begin{proof}
Note $\partial\stackrel{\circ}{D}=\overline{\stackrel{\circ}{D}}\cap(\stackrel{\circ}{D})^{c}$
and $\partial D=\overline{D}\cap\overline{D{}^{c}}$. Therefore it
is enough to show $(\stackrel{\circ}{D})^{c}\subset\overline{D^{c}}$.
Indeed if $y\in(\stackrel{\circ}{D})^{c}$ and $U$ is an open set
such that $y\in U$, then $U\cap D^{c}=\emptyset$ would imply $y\in U\subset D$
which would imply $y\in U\subset\stackrel{\circ}{D}$, contradicting
$y\in(\stackrel{\circ}{D})^{c}$. \end{proof}
\begin{thm}
\label{thm:Affirmative Answer}Let $(X,T)$ be a t.d.s. $(X,T)$ has
a zero-dimensional isomorphic extension iff $(X,T)$ has the small
boundary property.\end{thm}
\begin{proof}
Assume $(X,T)$ has SPB. For every $n\in\mathbb{N}$ one can choose
a partition of $X$ (into measurable sets), $\mathcal{P}^{n}={\{P_{1}^{n},\ldots,P_{m_{n}}^{n}\}}$
so that $diam(P_{k}^{n})<\frac{1}{n}$ and $ocap(\partial P_{k}^{n})=0$
for all $k$. Let $M$ denote the following \textit{symbolic matrix
system} (this terminology is introduced in Section $4$ of \cite{Dz01}):
Any element $m\in M$ consists of a two dimensional infinite matrix
$m=[m(n,i)]_{n,i\in\mathbb{N}}$, where for each $n$ the sequence
$m_{n}=[m(n,i)]_{i\in\mathbb{N}}$, which is called the \textit{$n$-th
row of $M$}, is an element in ${\{1,\ldots,m_{n}\}^{\mathbb{\mathbb{N}}}}$.
$M$ is equipped with product topology which makes it a compact metric
space. One defines the following left-shift action $S:M\rightarrow M$
by $Sm=[m(n,i+1)]_{n,i\in\mathbb{N}}$. For any $x\in X$, we associate
an element $m_{x}\in M$, given by the formula $m_{x}(n,i)=\mathcal{P}^{n}(T^{i}x)$,
where $\mathcal{P}^{n}(T^{i}x)=k$ iff $T^{i}x\in P_{k}^{n}$. Let
$Z={\overline{\{m_{x}|\, x\in X\}}}\subset M$. Notice $Z$ is closed
and $S$-invariant. There is a natural morphism $\pi:(Z,S)\rightarrow(X,T)$.
Indeed for any $z=[z(n,i)]_{n,i\in\mathbb{N}}\in Z$ define $\pi(z)=\bigcap_{n,i\in\mathbb{N}}T^{-i}\overline{P_{z(n,i)}^{n}}$.
As $diam(P_{k}^{n})<\frac{1}{n}$ for all $n,k$ the intersection
can be at most one point. For $x\in X$, $\pi(m_{x})=x$ and by choosing
$x_{k}\rightarrow z$ for any $z\in Z$, we see that the intersection defining $\pi(z)$ is non-empty
for all $z\in Z$. $Z$ is clearly zero-dimensional. We claim $\pi:(Z,S)\rightarrow(X,T)$
is an isomorphic extension. Fix $\mu\in M_{S}(Z)$. Let $O={\{}x\in X|\ |\pi^{-1}(x)|>1\}$.
It is enough to show $\pi_{*}\mu(O)=0$ because this implies $\pi$
is injective on a set of $\mu$-measure $1$. Note that $O\subset\bigcup_{n\in\mathbb{N},1\leq i\leq m_{n}}T^{-i}\partial P_{i}^{n}$.
As $\pi_{*}\mu(\partial P_{i}^{n})=0$ for all $n,i$ (see Subsection \ref{sub:The-Small-Boundary}), the result follows.

Now Assume $\pi:(Z,S)\rightarrow(X,T)$ is a zero-dimensional isomorphic
extension. Let $Clop(Z)$ be the collection of clopen sets of $Z$.
Let $\mathcal{U}~=~\{\stackrel{\circ}{\pi(D)}~\}_{D\in Clop(Z)}$. We
claim $\mathcal{U}$ is a basis. Indeed let $U\subset X$ be open
and $x\in U$. Using normality of $X$ choose $x\in V\subset\overline{V}\subset U$,
so that $V$ is open. For any $z\in\pi^{-1}(\overline{V})$ choose
a clopen set $B_{z}$, such that $z\in B_{z}\subset\pi^{-1}(U)$.
Choose a finite subcover of $\pi^{-1}(\overline{V})$, $B_{z_{1}},\mbox{\ensuremath{B_{z_{2}},}\ensuremath{\ldots},\ensuremath{B_{z_{n}}}and define \ensuremath{D=\bigcup B_{z_{i}}}}$.
Notice $\pi^{-1}(\overline{V})\subset D\subset\pi^{-1}(U)$ which
implies $\overline{V}\subset\pi(D)\subset U$ which implies $x\in V\subset\stackrel{\circ}{\pi(D)}\subset U$.

We now show that the members of $\mathcal{U}$ have small boundaries.
Let $D\in Clop(Z)$. As $\partial\stackrel{\circ}{\pi(D)}$ is closed
it is enough to show for every $\mu\in M_{S}(Z)$, $\pi_{\ast}\mu(\partial\stackrel{\circ}{\pi(D)})=0$.
As $\pi:(Z,S,\mathcal{B}_{Z},\mu)\rightarrow(X,T,\mathcal{B}_{X},\pi_{\ast}\mu)$
is an isomorphism, it is enough to show $x\in\partial\pi(D)\Rightarrow|\pi^{-1}(x)|>1$.
Indeed notice $\pi(D)\cup\pi(D^{c})=X$. Conclude $\pi(D)^{c}\subset\pi(D^{c})$.
As $x\in\partial\pi(D)=\pi(D)\cap\overline{\pi(D)^{c}}$ in particular,
$x\in\overline{\pi(D)^{c}}\subset\pi(D^{c}),$ which implies $\pi^{-1}(x)\cap D^{c}\neq\emptyset$,
i.e. $|\pi^{-1}(x)|>1$ (as also $\pi^{-1}(x)\cap D^{c}\neq\emptyset)$.
\end{proof}

\appendix
\renewcommand{\thesection}{\hspace*{-0pt}}
\section{}
\renewcommand{\thesection}{\Alph{section}}
\renewcommand{\thelem}{\Alph{lemma}}
\begin{lem}
\label{lem:Existence_simultaneous_good_segment} Let $A\subset X$,
$M\in\mathbb{N}$ be an even integer and $n:\bigcup_{i=-4M}^{\frac{M}{2}-1}T^{i}A\rightarrow\mathbb{R}$
a function. Assume there are $x_{1},x_{2}\in A$ so that for each
$x_{i}$, $i=1,2$ there is at most one index (depending on $x_{i})$
$-4M\leq j_{i}\leq\frac{M}{2}-2$ for which $n(T^{j{}_{i}+1}x_{i})\neq n(T^{j{}_{i}}x_{i})+1$.
Then one can find an index $-4M\leq r\leq0$ so that $\lfloor n(T^{r}x_{i})\rfloor\,\mod M\leq\frac{M}{2}$
and for $r\leq s\leq r+\frac{M}{2}-2$, $i=1,2$:

\begin{equation}
(\lceil n(T^{s}x_{i})\rceil\,\mod M)=(\lceil n(T^{r}x_{j})\rceil\,\mod M)+s-r\label{eq:round up mod}
\end{equation}

\begin{equation}
(\lfloor n(T^{s}x_{i})\rfloor\,\mod M)=(\lfloor n(T^{r}x_{j})\rfloor\,\mod M)+s-r\label{eq:round down mod}
\end{equation}
\end{lem}
\begin{proof}
By the proof of Lemma 5.7 of \cite{L99} one can find an index $-4M\leq r\leq0$
so that for $r\leq s\leq r+\frac{M}{2}-1$ one has for $i=1,2$:

\[
(n(T^{s}x_{i})\,\mod M)=(n(T^{r}x_{j})\,\mod M)+s-r
\]

\noindent In particular for $r\leq s\leq r+\frac{M}{2}-2$, $(n(T^{s}x_{i})\,\mod M)\in[0,\frac{M}{2}+s-r+1)\subset[0,M-1)$
and $(\lfloor n(T^{r}x_{i})\rfloor\,\mod M)\leq\frac{M}{2}$. We therefore
conclude (\ref{eq:round up mod}) and (\ref{eq:round down mod}) hold
for this range of indices.
\end{proof}
The following Lemma is closely related to Lemma 6.5 of \cite{L99}.
\begin{lem}
\label{lem:F not in V_n} Let $\epsilon>0$. Let $N,n,d,S\in\mathbb{N}$
with $N>2S$. Let $\gamma$ be an open cover of $X$ with $ord(\gamma)+1\leq(S-\frac{n}{2})d$.
Assume $\{q_{U}\}_{U\in\gamma}$ is a collection of distinct points
in $X$ and $\tilde{v}_{U}\in([0,1]^{d})^{N}$ for every $U\in\gamma$,
then there exists a continuous function $F:X\rightarrow([0,1]^{d})^{N}$,
with the following properties: \end{lem}
\begin{enumerate}
\item $\forall U\in\gamma$, $||F(q_{U})-\tilde{v}_{U}||_{\infty}<\frac{\epsilon}{2}$\label{enu:approximately v_U},
\item $\forall x\in X$, $F(x)\in co\{F(q_{U})|\, x\in U\in\gamma\}$\label{enu:convex combination},
\item \label{enu:not in V_n-1}For any $0\leq l<N-2S$, and $\lambda\in(0,1]$
and $x_{0},x_{1}\in X$ with $d(x_{0},x_{1})>\max_{W\in\gamma}diam(W)$
it holds:
\begin{equation}
(1-\lambda)F(x_{0})|_{l}^{l+2S-1}+\lambda F(x_{1})|_{l+1}^{l+2S}\notin V_{2S}^{n}\label{eq:in lemma: not in V_n}
\end{equation}
\noindent where,
\end{enumerate}
\[
V_{2S}^{n}\triangleq\{y=(y_{0},\ldots,y_{2S-1})\in([0,1]^{d})^{2S}|\quad\forall0\leq a,b\leq2S-1,\ (a=b\mod n)\rightarrow y_{a}=y_{b}\}.
\]

\begin{proof}
Let $\{\psi_{U}\}_{U\in\gamma}$ be a partition of unity subordinate
to $\gamma$ so that $\psi_{U}(q_{U})~=~1$. Let $\vec{v}_{U}\in([0,1]^{d})^{N}$,
$U\in\gamma$ be vectors that will be specified later. Define:

\[
F(x)=\sum_{U\in\gamma}\psi_{U}(x)\vec{v}_{U}
\]
For $x\in X$ define $\gamma_{x}=\{U\in\gamma|\,\psi_{U}(x)>0\}$.
Let $\lambda_{0}=1-\lambda$, $\lambda_{1}=\lambda$. Write (\ref{eq:in lemma: not in V_n})
explicitly as:

\begin{equation}
\sum_{j=0}^{1}\sum_{U\in\gamma_{x_{j}}}\lambda_{j}\psi_{U}(x_{j})\vec{v}_{U}|_{l+j}^{l+2S-1+j}\notin V_{2S}^{n}\label{explicit equality}
\end{equation}
Note that $dim(V_{2S}^{n})=nd$ and the left-hand side of (\ref{explicit equality})
is a convex combination of at most $2(ord(\gamma)+1)$ vectors. Note
$d(x_{1},x_{2})>\max_{W\in\gamma}diam(W)$ implies $\gamma_{x_{1}}\cap\gamma_{x_{2}}=\emptyset$.
As $2(ord(\gamma)+1)+nd\leq2Sd$ then by Lemma A.6 in \cite{Gut12a},
almost surely in $\big([0,1]^{d})^{2S}\big)^{2(ord(\gamma)+1)}$,
(\ref{explicit equality}) holds. As there is a finite number of constraints
of the form (\ref{explicit equality}), we can choose $\vec{v}_{U}\in([0,1]^{d})^{N}$,
$U\in\gamma$ so that property (\ref{enu:approximately v_U}) holds.
Finally property (\ref{enu:convex combination}) holds trivially as
$F(q_{U})=\vec{v}_{U}$.

\nocite{A}
\nocite{J74}
\nocite{LT12}
\nocite{GutTsu12}
\end{proof}
\bibliographystyle{alpha}
\bibliography{universal_bib}

\vspace{0.3cm}

\address{Yonatan Gutman, Department of Mathematics and Mathematical Statistics,
Centre for Mathematical Sciences, University of Cambridge, Wilberforce
Road, Cambridge CB3 0WA, UK \& Institute
of Mathematics, Polish Academy of Sciences, ul. \'{S}niadeckich 8, 00-956
Warszawa, Poland.}
\textit{E-mail address}: \texttt{y.gutman@dpmms.cam.ac.uk, y.gutman@impan.pl}
\end{document}